\documentclass[11pt,reqno]{amsart}
\usepackage{graphicx,tikz}
\usepackage{amsfonts,amsmath,amssymb}
\usetikzlibrary{matrix,arrows}
\newcommand{\cx}{{\mathbb{C}}}

\newcommand{\G}{\Gamma}

\newcommand{\D}{\mathbb D}
\newcommand{\C}{\mathbb C}
\newcommand{\R}{\mathbb R}
\newcommand{\N}{\mathbb N}
\newcommand{\PP}{\mathbb P}

\newtheorem{theorem}{Theorem}[section]
\newtheorem{lemma}[theorem]{Lemma}
\newtheorem{prop}[theorem]{Proposition}
\newtheorem{cor}[theorem]{Corollary}

\theoremstyle{definition}
\newtheorem{defn}[theorem]{Definition}
\newtheorem{rem}[theorem]{Remark}
\newtheorem{ex}[theorem]{Example}

\begin{document}
\title{Germs of singular Levi-flat  hypersurfaces and  holomorphic foliations}
\author{Rasul Shafikov* and Alexandre Sukhov**}
\begin{abstract}

It is shown that the Levi foliation of a real analytic Levi-flat hypersurface
extends to a $d$-web near a nondicritical singular point and admits a 
multiple-valued meromorphic first integral.
\end{abstract}

\maketitle

\let\thefootnote\relax

* Department of Mathematics, the University of Western Ontario, London, Ontario, N6A 5B7, Canada,
e-mail: shafikov@uwo.ca. The author is partially supported by the Natural Sciences and Engineering 
Research Council of Canada.

**Universit\'e des Sciences et Technologies de Lille, 
U.F.R. de Math\'ematiques, 59655 Villeneuve d'Ascq, Cedex, France,
e-mail: sukhov@math.univ-lille1.fr. The author is partially supported by Labex CEMPI.

\section{Introduction}

This paper is concerned with local properties of real analytic Levi-flat hypersurfaces near singular points. Levi-flat hypersurfaces
arise naturally in the theory of holomorphic foliations and differential equations, in particular, in the study of minimal sets for foliations. They were studied recently by several authors 
from different points of view (see, e.g., \cite{Be,Bru1,Bru2,BuGo,CerNeto,CerSad,Le}). 

Let $z=(z_1, \dots, z_n)$, $z_j = x_j + i y_j$, be the standard coordinates in $\mathbb C^n$.
A {\it real analytic hypersurface} $\Gamma$ in a domain $\Omega \subset \C^n$ is a closed real analytic subset of $\Omega$ of dimension 
$2n-1$. This means, in particular, that for every point $q \in \Omega$ there exists an open neighbourhood $U$ of $q$ in $\Omega$ and a 
real analytic function $\rho:U \to \R$ such that 
\begin{eqnarray}
\label{hypeq}
\Gamma \cap U = \rho^{-1}( 0 ). 
\end{eqnarray}
We say that $\Gamma$ is {\it real algebraic} if $\rho$ is a real polynomial. It is well known that the regular part of a Levi-flat 
hypersurface $\G$, which we denote by $\G^*$, is foliated by complex hypersurfaces forming the so-called {\it Levi foliation} $\mathcal L$. 
Global or local extension of this foliation to the ambient space is an important question, see, e.g.,~\cite{Bru1,Bru2,CerNeto,Fe} 
for recent results in this direction. Brunella~\cite{Bru1} gave an example of a Levi-flat hypersurface in $\cx^2$, singular at the origin, 
such that the Levi foliation 
extends to a neighbourhood of the origin as a {\it singular web}, but not as a foliation, see Example~\ref{e.B}. In this paper we 
are interested in finding general sufficient conditions for the extension of the Levi foliation $\mathcal L$ as a singular holomorphic $d$-web near 
a singular point. Section~\ref{s.webs} is dedicated to a detailed discussion of webs, but loosely speaking, singular holomorphic webs can 
be viewed as foliations with branching. A singular point $p$ of a Levi-flat hypersurface is called {\it dicritical} if  infinitely many leaves of the 
Levi foliation have $p$ in their closure. Our  first result is the following.

\begin{theorem}\label{t.1}
Let $\G\subset \Omega$ be an irreducible Levi-flat real analytic hypersurface in a domain $\Omega\subset \cx^n$,
$n\ge 2$, and $0\in \overline{\G^*}$. Assume that at least one of the following conditions holds:
\begin{itemize}
\item[(a)] $0\in \G$ is not a dicritical singularity. 
\item[(b)] $\G$ is a real algebraic hypersurface.
\end{itemize}
Then there exist a neighbourhood $U$ of the 
origin and a singular holomorphic $d$-web $\mathcal W$ in $U$ such that $\mathcal W$ extends the Levi foliation $\mathcal L$. 
Furthermore,  $\mathcal W$ admits a multiple-valued meromorphic first integral in $U$.
\end{theorem}

We say that a web $\mathcal W$ extends the foliation $\mathcal L$ if every leaf of $\mathcal L$ is a leaf of $\mathcal W$.
By a first integral of $\mathcal W$ we mean a meromorphic correspondence which is constant along the leaves of $\mathcal W$, see
Section~\ref{s.webs} for details. We note that under some additional assumptions on the singular locus of $\Gamma$, part (a) of 
our result was obtained recently by Fern\'andez-P\'erez~\cite{Fe}. Example~\ref{e.s} below gives a singular Levi-flat hypersurface which satisfies 
condition (a) of Theorem~\ref{t.1} but not those of~\cite{Fe}.  

Our approach is rather constructive, especially under condition (b) in Theorem~\ref{t.1}. In many cases 
one can write down explicitly the $d$-web that gives the extension of the Levi foliation, see Section~\ref{s.examples} for
relevant examples. The key point of our approach lies in the connection between singular webs and first order analytic partial differential equations,
although we do not claim any particular originality here. Presently, the most commonly used definition of webs is through the geometry the 
projectivized cotangent bundle. We reconstruct the connection between geometry of singular webs and analytic PDEs through compactification of 
the 1-jet bundle of functions on $\cx^{n-1}$ and its identification with the projectivized cotangent bundle of $\cx^n$, see
Section~\ref{s.webs} for details. 

We also have another version of Theorem \ref{t.1} corresponding to an alternative way to describe Levi-flat hypersurfaces. 
Let $\Omega$ be a neighbourhood of the origin in $\C^n$. Denote by ${\mathbb S} = \{ \zeta = e^{i\tau}: \tau \in [0,2\pi] \}$ 
the unit circle in $\C$. Let $H:\Omega \times {\mathbb S} \longrightarrow \C$ be a real analytic (complex-valued) function 
which is holomorphic in the variable $z \in \Omega$ for every $\zeta \in {\mathbb S}$. Consider its zero locus
\begin{eqnarray}
\label{subanalytic0}
\hat\Gamma = H^{-1}(0) ,
\end{eqnarray}
which is a real analytic subset of $\Omega \times \mathbb S$. Suppose that the function $H$ has the maximal rank 2 on an open dense subset of $\hat\Gamma$. Consider now the set
\begin{eqnarray}\label{subanalytic1}
\Gamma = \{ z \in \Omega: H(z,\zeta) = 0 \, \,\,\, \mbox{for some} \,\,\,\, \zeta {\in \mathbb S} \} .
\end{eqnarray} 
In general $\Gamma$ is only a subanalytic subset of $\Omega$ since it coincides with the image of $\hat\Gamma$ under the canonical projection into $\Omega$. We call such $\Gamma$ a {\it Levi-flat subanalytic hypersurface} in $\Omega$. We say that $\Gamma$ defined by (\ref{subanalytic1}) is {\it algebraically  parametrized} if the function $H$ admits representation
\begin{eqnarray}
\label{rational}
H(z,\zeta) = \sum_{k=-N}^N h_k(z)\zeta^k, \ \ \zeta \in \mathbb S, \ \ h_j \in \mathcal O(\Omega),
\end{eqnarray}
and $H(\cdot, \zeta)$ is not a constant.  We use the same notation for the regular part of $\Gamma$ and the Levi foliation as in the real analytic case. Our second result is the following

\begin{theorem}
\label{t.2}
Let $\G\subset \Omega$ be a   Levi-flat subanalytic hypersurface in a domain $\Omega\subset \cx^n$, $n\ge 2$. Assume that $\G$ is algebraically parametrized, and $0\in \overline{\G^*}$. Then there exist a neighbourhood $U$ of the 
origin and a singular holomorphic $d$-web $\mathcal W$ in $U$ such that $\mathcal W$ extends the Levi foliation $\mathcal L$. 
Furthermore,  $\mathcal W$ admits a multiple-valued meromorphic first integral in $U$.
\end{theorem}

The reason to consider Levi-flat hypersurfaces defined by (\ref{subanalytic1}) is that this gives a very convenient way to construct Levi-flat hypersurfaces. Indeed, in terms of  (\ref{hypeq}) the Levi-flatness means that $\rho$ is a solution of a non-linear PDE operator of Monge-Amp\`ere type. One can view (\ref{subanalytic1}) as an effective way to describe solutions of this operator.

\section{Background: Levi-flat hypersurfaces and Segre Varieties}

In this section we provide relevant background material on singular Levi-flat hypersurfaces and their Segre varieties, and 
establish some basic general facts concerning their geometry. 

\subsection{Levi-flat hypersurfaces}  The neighbourhood $U$ in (\ref{hypeq}) can be always chosen in the form of a polydisc 
\begin{eqnarray}
\label{polydisc}
\Delta(p,\varepsilon) = \{ z \in \C^n: \vert z_j - q_j \vert < \varepsilon \}
\end{eqnarray}
centred at $q$ and of radius $ \varepsilon > 0$. Without loss of generality we may assume that $q = 0$. Then for $\varepsilon$ small enough, the function $\rho$ admits a
 convergent in $U$ Taylor  expansion 
\begin{eqnarray}
\label{exp}
\rho(z,\overline z) = \sum_{IJ} c_{IJ}z^I \overline{z}^J, \ c_{IJ}\in\cx, \ \ I,J \in \N^n.
\end{eqnarray}
Note that coefficients $c_{IJ}$ satisfy the condition
\begin{eqnarray}
\label{coef}
\overline c_{IJ} = c_{JI},
\end{eqnarray}
which is imposed by the reality of $\rho$. In local questions we can 
always assume that  (\ref{hypeq}) is a global defining equation of $\Gamma$ in a neighbourhood $\Omega$ of the origin. We may 
shrink $\Omega$ if necessary when we are dealing with germs of real hypersurfaces. 

In this paper we  adopt the following terminology. A real analytic hypersurface $\Gamma$ is called {\it irreducible} in $\Omega$ if it cannot be 
represented as the union $\Gamma = \Gamma_1 \cup \Gamma_2$ of two real analytic hypersurfaces $\Gamma_j$ in $\Omega$. 
We call a point $q \in \G$ a {\it regular} point, if $\G$ is a real analytic submanifold of dimension $2n-1$ in a neighbourhood of $q$, i.e., 
a smooth analytic hypersurface near $q$. The union of all regular points form a regular locus denoted by $\G^*$. By definition of a 
Levi-flat hypersurface, $\G^*$ is an open non-empty subset of $\G$. Its complement 
$\Gamma^{sng}:= \Gamma \setminus \Gamma^{*}$ is called the {\it singular locus} of $\Gamma$. 
Note that this convention is different from the usual definition of a regular  point in the semi-analytic or subanalytic geometry. There, 
a similar notion is less restrictive and a real analytic set is allowed to be a submanifold of {\it some} dimension near a regular point. 
The classical example of the Whitney umbrella (see, for instance, \cite{Lo}) shows that an irreducible real analytic set does not always 
have pure dimension. By our definition, points of $\Gamma$, where $\Gamma$ is a submanifold of dimension smaller than $2n-1$, 
belong to the singular locus. For that reason $\Gamma^*$ may be not dense in $\Gamma$. Note that $\G^{sng}$ is a closed 
semi-analytic subset of $\Gamma$ (possibly empty) of real dimension at most $2n-2$. In what follows we always assume that
the hypersurface $\G$ that we consider is irreducible.

Let $\Gamma$ be a real analytic hypersurface in $\Omega$. If $q$ is a regular point of $\Gamma$, then there exist a neighbourhood
$U$ of $q$ and a function $\rho$ real analytic in $U$ such that \eqref{hypeq} holds, and the gradient of $\rho$ does not vanish: 
$\nabla \rho \ne 0$ in $U$.
For $q\in \G^*$ consider the complex tangent space $H_q(\Gamma):= T_q(\Gamma) \cap JT_q(\Gamma)$. Here $J$ denotes the standard complex structure of $\cx^n$. The {\it Levi form} 
of $\Gamma$ is a Hermitian quadratic form defined on $H_q(\Gamma)$ by
$$
L_q(v) = \sum_{k,j} \frac{\partial^2\rho}{\partial z_k \partial \overline{z}_j}(q) v_k \overline{v}_j, \,\,\,\,\,v \in H_q(\Gamma).
$$
A real analytic hypersurface $\Gamma$ is called {\it Levi-flat} if its Levi form vanishes identically at every regular point of $\Gamma$. 
It is well known that for every point  $q \in \Gamma^*$, there exists a local biholomorphic change of coordinates centred at $q$ such 
that in the new coordinates $\Gamma$ in some neighbourhood $U$ of  $q = 0$ has the form 
\begin{eqnarray}
\label{NormForm}
\{ z \in U: z_n + \overline{z}_n = 0 \} .
\end{eqnarray}
Hence, $\Gamma \cap U$ is foliated by complex hyperplanes $\{z_n = c, \ c \in \mathbb R\}$. This foliation is called {\it the Levi 
foliation} of $\Gamma^*$. Every leaf is tangent to the complex tangent space of $\Gamma^*$. Clearly, it extends as a holomorphic 
codimension one foliation to a neighbourhood of $q$ in the ambient space : it suffices to allow $c$ to be a complex constant. 
In general, such simple representation of a Levi-flat hypersurface $\Gamma$ does not exist near 
singular points. The Levi-foliation of $\G^*$  will be denoted  by~$\mathcal L$. 

\subsection{Segre varieties} 
Another important tool used in the paper is the family of the so-called {\it Segre varieties} associated with a real analytic 
hypersurface~$\Gamma$. For a function $\rho$ 
with the expansion (\ref{exp}) this family  is defined using the {\it complexification} of $\rho$ given by 
\begin{eqnarray}\label{compl}
\rho(z,\overline w) = \sum_{IJ} c_{IJ}z^I\overline{w}^J ,
\end{eqnarray}
i.e., we replace the variable $\overline z$ with an independent variable $\overline w$.
We assume that the neighbourhood $U$ of the origin in $\C^{N}$ is chosen so small that the series (\ref{compl}) converges for all $z,w \in U$. 
Then $\rho(z,\overline w)$ is holomorphic in $z \in U$ and antiholomorphic in $w \in U$. In view of (\ref{coef}) one has
\begin{eqnarray}
\label{conj}
\rho(z,\overline w) = \overline{\rho(w,\overline{z})}, \ \forall (z,w) \in U \times U .
\end{eqnarray}

For $w \in U$ consider a complex analytic hypersurface given by
\begin{eqnarray}
\label{Segre1}
Q_w = \{ z \in U : \rho(z,\overline w) = 0 \} .
\end{eqnarray}
It is called the {\it Segre variety} of the point $w$ associated with $\Gamma$. The following properties of Segre varieties are well known, see, 
for instance, \cite{DiPi}.

\begin{prop}\label{SegreProp}
Let $\G$ be a real analytic hypersurface in $\mathbb C^n$, $n>1$. Then
\begin{itemize}
\item[(a)] $z \in Q_z \Longleftrightarrow z \in \Gamma$,
\item[(b)] $z \in Q_w \Longleftrightarrow w \in Q_z$,
\item[(c)] (invariance property) Let $\Gamma$, $\Gamma'$ be real analytic hypersurfaces, 
$q \in \Gamma^*$, $q' \in (\Gamma')^*$, and $U \ni q$, $U' \ni q'$ be small neighbourhoods,  and let
$f:U \to U'$ be a holomorphic map such that $f(\Gamma \cap U) \subset \Gamma' \cap U'$. Then
$$f(Q_w) \subset Q'_{f(w)}$$
for all $w \in U$. In particular, if $f:U \to U'$ is biholomorphic, then $f(Q_w) = Q'_{f(w)}$.
Here $Q_w$ and $Q'_{f(w)}$ are Segre varieties associated with $\Gamma$ and $\Gamma'$ respectively.
\end{itemize}
\end{prop}

Property (c) can be viewed as the biholomorphic invariance of Segre varieties. It has important consequences. For example, it 
allows us to view intrinsically the described above phenomenon of extension of the Levi foliation to a  holomorphic foliation 
in a neighbourhood of a regular point of~$\Gamma$. Indeed, the complex hyperplanes $\{ z_n = c \}$ in coordinates 
(\ref{NormForm}) are precisely Segre varieties of $\Gamma$ for every complex $c$.

Let $q \in \Gamma^*$. Denote by ${\mathcal L}_q$ the leaf of the Levi foliation through $q$.
Note that by definition it is a connected complex hypersurface closed in $\Gamma^*$.
As a simple consequence of Proposition \ref{SegreProp} we have 

\begin{cor}
\label{Segre+Levi}
Let  $a \in \Gamma^*$. Then the following holds:
\begin{itemize}
\item[(a)]  The leaf ${\mathcal L}_a$ is contained in the unique irreducible component $S_a$ of $Q_a$. In a small neighbourhood of $a$ this is also a unique 
complex hypersurface through $a$ which is contained in $\Gamma$.
\item[(b)] For every $a \in \Gamma^*$ the complex hypersurface $S_a$ is contained in $\Gamma$;
\item[(c)] For every $a, b \in \Gamma^*$ one has $b \in S_a \Longleftrightarrow S_a = S_b$.
\item[(d)] Suppose that $a \in \Gamma^*$ and ${\mathcal L}_a$ touches a point  $q \in \Gamma$ such that $\dim_\C Q_q = n-1$ (the point $q$  may be singular). Then  $Q_q$ contains $S_a$ as an irreducible component. 
\end{itemize}
\end{cor}

\begin{proof} (a)  In view of the invariance of the Levi form under biholomorphic maps, the Levi foliation is an intrinsic notion, i.e., it is
independent of the choice of (local) holomorphic coordinates. Similarly, in view of the biholomorphic invariance of Segre varieties described in 
Proposition~\ref{SegreProp}(c), these are also intrinsic objects. There exist a small neighbourhood $U'$ of $a$ and a holomorphic map  
which takes $a$ to the origin and is one-to-one between $U'$ and a neighbourhood $U$ of the origin, such that the image of $\Gamma$ has 
the form (\ref{NormForm}). Hence, without loss of generality we may assume  that $a = 0$ and view (\ref{NormForm}) as a representation of  
$\Gamma \cap U$ in the above local coordinates. Then $Q_0 \cap U = \{ z_n = 0 \}$. Hence, going back to the initial coordinates, we obtain 
by the invariance of Segre varieties that the intersection $Q_a \cap U'$ is a smooth complex hypersurface in $\Gamma \cap U'$ which coincides 
with ${\mathcal L}_a \cap U'$. In particular, it belongs to a unique irreducible component of $Q_a$. It follows also from (\ref{NormForm}) 
that it is a unique complex hypersurface through $a$ contained in a neighbourhood of $a$ in $\Gamma$.

(b) Recall that we consider $\Gamma$ with a defining function $\rho$ admitting expansion (\ref{exp}) in a polydisc of the form (\ref{polydisc}) 
centred at the origin. Since $S_a$ is contained in $\Gamma$ near $a$, it follows by analyticity of $\rho$ and uniqueness that 
$\rho \vert_{S_a} \equiv 0$, i.e., $S_a$ is contained in $\Gamma$.

(c) By part (b) the complex hypersurface $S_a$ is contained in $\Gamma$. Therefore, in a small neighbourhood of  $b$ we have $S_a = S_b$ 
by part (a). Then also globally $S_a = S_b$ by the uniqueness theorem for irreducible complex analytic sets.

(d) By assumption, the holomorphic function $z \mapsto \rho(z,\bar q)$ does not vanish identically. Consider a sequence of points 
$q^m \in {\mathcal L}_a$ converging to $q$. It follows by (c) that the complex hypersurface  $S_a = S_{q^m}$ is independent of $m$  
and by (a) $S_{q^m}$ is an irreducible component of $Q_{q^m}$. Passing to the limit we obtain that  $S_a$ is contained in $Q_q$ as an 
irreducible component.
\end{proof}

Let $\G$ be a real analytic Levi-flat hypersurface in $\C^n$. A singular point $q\in \G$ is called 
{\it dicritical} if $q$ belongs to  infinitely many geometrically different leaves ${\mathcal L}_a$. It follows by Corollary ~\ref{Segre+Levi}(d)  
that a point $q$ is dicritical if and only if  $\dim_\C Q_q = n$. Singular points which are not dicritical are called {\it nondicritical}. 

\begin{lemma}
Dicritical singular points form a complex analytic subset of $\Gamma$ of complex dimension at most $n-2$, in particular, it is a discrete set if 
$n=2$. If $\Gamma$ is algebraic, then the set of dicritical singularities is also complex algebraic.
\end{lemma}
\begin{proof} Indeed, the complexification of the defining function $\rho$ of $\G$ given by~\eqref{compl} defines (after additional complex 
conjugation) a complex analytic set $\G^c = \{(z,w)\in U_z \times U_w : \rho(z,w)=0\}$, where  $U_z$ and $U_w$ are suitable
neighbourhoods of the origin. The fibres of the projection 
$\pi_w : \G^c \to U_w$ can be identified with the Segre varieties $Q_w$. Therefore, dicritical singularities of $\G$ correspond to points $w$
such that $\dim \pi_w^{-1}(w)= n$. Denote by $l_{(z,w)} \pi_w$ the germ of the fibre $\pi_w^{-1}(w)$ at a point $(z,w)\in \G^c$.
Since the map $(z,w) \to \dim l_{(z,w)} \pi_w$ is upper semicontinuous, the set of dicritical singularities is closed.
By the Cartan-Remmert theorem (see, e.g., \cite[V.3.3]{Lo1}), the set $\Sigma = \{ (z,w) : \dim l_{(z,w)} \pi_w \ge n\}$
is complex analytic. The set of dicritical singularities can now be identified with $\pi_w(\Sigma)$.
Since $\dim l_{(z,w)} \pi_w \le n$, by Remmert's rank theorem (\cite[V.6]{Lo1}), $\pi_w(\Sigma)$ is a complex
analytic subset of $U_w$, after possibly shrinking $U_z$ and $U_w$. Finally, since $\G^c$ is irreducible of dimension $2n-1$, 
the set $\pi_w(\Sigma)$ has dimension at most $n-2$, which proves the assertion.

If $\Gamma$ is algebraic, then $\G^c\subset \cx P^n \times \cx P^n$, and $\Sigma\subset \G^c$ are also algebraic.
It follows that $\pi_w(\Sigma)$ is algebraic as well.
\end{proof}

\section{Singular webs}\label{s.webs}
In this section we define singular holomorphic webs and outline the connection between webs and
differential equations. This connection is transparent in dimension two, so we will discuss this case
separately. For a comprehensive treatment of singular webs see, e.g.,~\cite{Pi}.

\subsection{Webs in $\mathbb C^2$}\label{s.webinc2}
Recall that the germ of a holomorphic codimension one foliation $\mathcal F$ in $\mathbb C^n$, $n\ge 2$, can 
be given by the germ of a holomorphic $1$-form $\omega\in \Lambda^1(U)$ satisfying the Frobenius integrability 
condition $\omega \wedge d\omega =0$. The leaves of $\mathcal F$ are then complex hypersurfaces $L$ that 
are tangent to $\ker \omega$. The foliation $\mathcal F$ is {\it singular} if the set 
$\mathcal F^{\rm sng} =\{z : \omega(z) =0\}$ is nonempty and of codimension at least 2. 

In dimension $2$ the integrability condition for $\omega$ always holds, and  the above definition of a 
(nonsingular) foliation can be interpreted in the following way: for a suitably chosen open set $U$ and
coordinate system in $\cx^2$ the foliation $\mathcal F$ is given by a {\it holomorphic first order ODE} 
\begin{equation}\label{e.fol}
\frac{dz_2}{dz_1}=F(z_1, z_2)
\end{equation}
with respect to unknown function $z_2=z_2(z_1)$. The leaves of the foliation $\mathcal F$ are then the 
graphs of solutions of the ODE. This interpretation admits a far reaching generalization which we now 
describe. Our considerations are local and should be understood on the level of germs,
but to simplify the discussion we will work with appropriate representatives of the germs.

Let $U_1,U_2$ be  domains in $\C$ containing the origin. Set  $U = U_1 \times U_2 \subset \cx^2$,  and 
consider a holomorphic function $\Phi$  on $U \times \C$. It defines a {\it holomorphic ordinary differential 
equation} on $U \times \C$, 
\begin{eqnarray}\label{ME1}
\Phi(z_1,z_2,p) = 0
\end{eqnarray}
with $z = (z_1,z_2) \in U$ and $p= \frac{dz_2}{dz_1} \in \C$. This is an equation for the unknown function 
$z_2 = z_2(z_1)$; in other words, we view $z_1$ and $z_2$ as the independent and the dependent variables respectively. 
 For $d \in \mathbb N$, a {\it singular holomorphic $d$-web} $\mathcal W$ in $U$ is defined by equation (\ref{ME1})  
 where $\Phi$ is of the form
\begin{eqnarray}\label{ME2}
\Phi(z,p) = \sum_{j=0}^d \Phi_j(z) p^{j}.
\end{eqnarray}
In general, there are $d$ families of solutions of~\eqref{ME1} (with $\Phi(z,p)$ as in~\eqref{ME2}), which are either unrelated to each other or may fit together along some complex curves (branching). The graphs of solutions are called the {\it leaves} of 
$\mathcal W$.  

\begin{ex}\label{ex.brunella}
Consider the ODE of the form $p^2= 4z_2$ in $\mathbb C^2$. Its solutions form a complex one-dimensional family 
of curves $L_c =\{z_2 = (z_1 + c)^2\}$, $c\in \cx$. For every point $b = (b_1, b_2)\in \cx^2$ with $b_2\ne 0$,
there exist exactly two curves passing through this point, namely, $L_{-b_1-\sqrt b_2}$ and $L_{-b_1+\sqrt b_2}$
(we can take an arbitrary branch of $\sqrt z$). These curves meet at $b$ transversely. But any point $(b_1, 0)$ is 
contained only in one curve $L_{-b_1}$ of the family. $\diamond$
\end{ex}

If $d=1$, then (\ref{ME1}) becomes resolved with respect to the derivative, so $1$-webs simply coincide with 
holomorphic foliations (possibly singular). If  \eqref{ME2} factors into distinct, linear in $p$ terms, i.e., $\Phi(z,p) = \Pi_{j=1}^d (p - f_j(z))$, 
where $f_j(z)$ are holomorphic functions, then each ODE $p = f_j(z)$ defines a holomorphic foliation $\mathcal F_j$.  
If the leaves of $\mathcal F_j$ intersect in general position (resp. pairwise transversely) then the union of $\mathcal F_j$ 
is called a smooth (resp. quasi-smooth) holomorphic $d$-web. Thus, our definition of a singular $d$-web is a proper generalization of smooth webs. From this point of view one can consider singular $d$-webs as a ``branched" version of 
their smooth counterparts.

\subsection{Webs in $\mathbb C^n$, $n\ge 2$}
The definition of a $d$-web (singular or smooth) via differential equations does not have a simple generalization to 
higher dimensions. There are several equivalent definitions in the literature. We will use a more geometric one that is
more suitable for our purposes. 

We denote by $\PP T^*_n:=\PP T^*\C^n$ the projectivization of the cotangent bundle of $\C^n$ with the 
natural projection $\pi: \PP T^*_n \to \C^n$. A local trivialization of $\PP T^*_n$ is isomorphic to 
$U \times G(1,n)$, where $U\subset \cx^n$ is an open set and $G(1,n) \cong \mathbb CP^{n-1}$ is the Grassmanian 
space of linear complex one dimensional subspaces in~$\cx^{n}$. The space $\PP T^*_n$ has the canonical
structure of a {\it contact manifold}, which can be described (using coordinates) as follows. Let $z = (z_1,...,z_n)$ be the
coordinates in $\C^n$ and $(\tilde p_1,..., \tilde p_n)$ be the fibre coordinates corresponding to the basis of differentials 
$dz_1, \dots, dz_n$.   We may view $[\tilde p_1, \dots, \tilde p_n]$ as homogeneous coordinates on $G(1, n)$.
Then in the affine chart $\{\tilde p_n \ne 0\}$, in nonhomogeneous coordinates $p_j = \tilde p_j / \tilde p_{n}$, 
$j=1,\dots, n-1$, the 1-form
\begin{equation}\label{e.cont}
\eta = d z_n + \sum_{j=1}^{n-1} p_j dz_j
\end{equation}
is a local contact form. Considering all affine charts $\{\tilde p_j \ne 0\}$ we obtain a global contact structure.

Let $U$ be a domain in $\C^n$. Consider a complex purely n-dimensional analytic subset $W$  in 
$\pi^{-1}(U) \subset \PP T^*_n$. Suppose that the following conditions hold:

\begin{itemize}
\item[(a)] the image under $\pi$ of every irreducible component of $W$ has dimension $n$;
\item[(b)] a generic fibre of $\pi$ intersects $W$ in $d$ regular (smooth) points and at every such point $q$ the differential 
$d \pi(q): T_qW \to \C^n$ is surjective;
\item[(c)] the restriction of the contact form $\eta$ on the regular part of $W$ is Frobenius integrable. So $\eta\vert_W = 0$ 
defines the foliation ${\mathcal F}_W$ of the regular part of $W$. (The leaves of the foliation ${\mathcal F}_W$ are called 
Legendrian submanifolds.)
\end{itemize}
Under these assumptions we define a {\it singular $d$-web} ${\mathcal W}$ in $U$ as a triple $(W,\pi,{\mathcal F}_W)$. 
A leaf of the web ${\mathcal W}$ is a component of the projection of a leaf of ${\mathcal F}_W$ into $U$. Note that at a 
generic point $z\in U$ a $d$-web $(W,\pi,{\mathcal F}_W)$ defines in $U$ near $z$ exactly $d$ families of smooth 
foliations. 

\subsection{Connection between webs and PDEs}
In this subsection we establish the connection between singular $d$-webs and PDEs in higher dimensions. This will be 
important in the proof of our main results. We need first to interpret a first order PDE as a subvariety of a 1-jet bundle. 
Recall that two smooth functions $\phi_1$ and $\phi_2$ have the same $k$-jet at a source point $x^0 \in \cx^n$ if 
$\vert \phi_1(x) - \phi_2(x) \vert = o(\vert x - x^0 \vert^k)$. In other words, this simply means that their Taylor 
expansions of order $k$ at $x^0$ coincide. The equivalence classes with respect to this relation are called $k$-jets 
at~$x^0$. 

Let $U\subset \mathbb C^{n-1}$ be a domain. Consider $J^1(U,\C)$, the space of $1$-jets of holomorphic functions
$f: U \to \mathbb C$. We can view such functions as sections of the trivial line bundle 
$U\times \cx \to U$. Then $J^1(U,\C)$ can be viewed as a  vector bundle 
\begin{equation}\label{e.Jproj}
\pi : J^1(U,\C) \to U \times \cx
\end{equation}
of rank $(n-1)$. Let $z'=(z_1,\dots,z_{n-1})$ be the coordinates on $U \subset \cx^{n-1}$, $z_n$ be 
the coordinate in the target space, and let $p_j$ denote the partial derivatives of $z_n$ with
respect to $z_j$. Then $(z,p)=(z_1, \dots, z_n, p_1, \dots, p_{n-1})$ form the coordinate system on $J^1(U,\C)$.
Note that $\dim J^1(U,\C)=2n-1$. The space $J^1(U,\C)$ admits the structure of a contact manifold with the 
contact form 
\begin{equation}\label{e.cont2}
\theta = d z_n - \sum_{j=1}^{n-1} p_j dz_j .
\end{equation}
Given a local section $f: U \to \cx$, let $j^1f : U \to J^1(U,\C)$, $j^1f: z \mapsto j^1_zf$ denote the 
corresponding section of the 1-jet bundle. Then a section $F: U \to J^1(U,\C)$ locally coincides with
$j_1f$ for some section $f: U \to \cx$ if and only if $F$ annihilates $\theta$. Now observe that 
the map $\iota : (z,p) \mapsto (z,-p)$ in the chosen coordinate systems is a biholomorphism whose pullback 
sends $\eta$ to $\theta$ in~\eqref{e.cont}, i.e., $\iota: J^1(U,\C) \to \PP T^*_n$ is a contactomorphism. 
Using the map $\iota$ we may view the projectivized  cotangent bundle $\PP T^*_n$ as a compactification 
of the 1-jet bundle. Alternatively, we may compactify $J^1(U,\C)$ in the variables $p$, that is, we compactify
every fibre $\cx^{n-1}_p$ to $\cx P^{n-1}$. Since the dependence of the form $\theta$ is linear in $p$, the 
compactified bundle will be a contact complex manifold.

Any first order holomorphic PDE of the form 
\begin{equation}\label{e.pde}
\Phi\left(z_1, \dots, z_{n-1}, z_n, \frac{\partial z_n}{\partial z_1}, \dots, \frac{\partial z_n}{\partial z_{n-1}} 
\right) =0
\end{equation}
with respect to the unknown function $z_n= z_n(z_1,\dots, z_{n-1})$ defines a complex hypersurface $W_\Phi$ 
in $J^1(U,\C)$ given by the equation $\Phi(z, p) =0$. Any solution $z_n = f(z_1, \dots, z_{n-1})$ of~\eqref{e.pde} 
admits {\it prolongation} to $J^1(U,\C)$, i.e., defines there an $(n-1)$-dimensional submanifold $S_f$ given by 
$$
\left\{z_n = f(z_1, \dots, z_{n-1}),\ \ p_j=\frac{\partial f}{\partial z_j}(z_1,\dots, z_{n-1}),\ \ j=1,\dots, n-1 \right\} .
$$
Hence, solutions of this differential equation can be identified with holomorphic sections $S_f$ of $W_\Phi$ annihilated  
by the contact form $\theta$. As an example, for equation~\eqref{e.fol}, the corresponding hypersurface 
$W \subset J^1(U,\cx)$, $U\subset \cx$, is simply the graph of a holomorphic function $p=F(z)$. It is foliated by graphs 
of solutions, which are integral curves of the distribution defined by $\theta$, and the corresponding foliation $\mathcal F$ 
in $\cx^2$ is obtained by the biholomorphic projection $\pi|_W: W \to \mathbb C^2_z$.

Suppose now that we have several differential equations of the form~\eqref{e.pde} such that the intersection of 
the corresponding hypersurfaces $W_\Phi$ is a complex analytic subset $W$ of $J^1(U,\cx)$ of pure dimension $n$. 
For example, we can have $n-1$ equations in general position. Suppose further that the compactification of $W$ 
in the projectivized cotangent bundle $\PP T^*_n U$ still forms a complex subvariety of the same dimension. This 
is the case, for example, if all $\Phi(z,p)$ are polynomial with respect to $p$ with coefficients holomorphic in $z$. 
Then $W$ satisfies the definition of a singular web given in the previous subsection. Note that we need to consider 
compactification of $W$ only if the projection in~\eqref{e.Jproj} has fibres of positive dimension, since otherwise 
the projection from $J^1(U,\cx)$ gives the same web in $U\subset \cx^n$.

Also note that for $n=2$ both definitions of a singular web agree. Indeed, given a differential equation~\eqref{ME1}, 
\eqref{ME2}, the function $\Phi(z,p)$ is polynomial in $p$ and thus it can be projectivized
to define a hypersurface in $\PP T^*_2 U$. This gives the hypersurface in $\PP T^*_2 U$ that has the required 
properties. Conversely, let $U$ be a neighbourhood of the origin in $\cx^2$, let $W$ be a complex hypersurface in 
$\PP T^*_2 U$ with the surjective projection  $\pi : W \to U$. Without loss of generality assume that $W$ is irreducible. 
If $\pi^{-1}$ is discrete, then by the Weierstrass preparation theorem, in a sufficiently small neighbourhood $\tilde U$
of the origin $W$ can be represented by a Weierstrass pseudo-polynomial in $p$, and we obtain the definition of the
web given in Section~\ref{s.webinc2}. Suppose that $\dim \pi^{-1}(0)=1$. Let 
$\tau : \cx^2_{(p_0,p_1)}\setminus \{0\} \to \mathbb CP^1$ be the natural projection given by 
$\tau(p_0, p_1) =[p_0,p_1]$. Let 
$$
\tilde \tau = ({\rm Id}, \tau): U \times (\cx^2\setminus \{0\}) \to U \times \mathbb CP^1 .
$$
Then the set $\tilde W = \tilde\tau^{-1}(W)$ is complex analytic in $U \times (\cx^2\setminus \{0\})$ of
dimension 3. The set $U \times\{0\}$ is removable, and so we may assume that $\tilde W$ is complex analytic
in $U\times \cx^2$. In a neighbourhood of $(0,0)$ it can be given by an equation $\phi(z,p)=0$. But since its
image is complex analytic in $U\times \cx P^1$, the function $\phi$ is in fact a homogenous polynomials in $p$.
This shows that in a neighbourhood of the origin in $U$, the hypersurface $W$ can be given by an equation which is 
polynomial in variable $p$, and we again recover the definition of Section~\ref{s.webinc2}.

\subsection{Meromorphic first integral}
We also need a related notion of a {\it multi-valued meromorphic first integral}.
Let $X$ and $Y$ be two complex manifolds and $\pi_X: X \times Y \to X$ and $\pi_Y :X \times Y \to Y$ be the natural 
projections. A $d$-valued {\it meromorphic correspondence} between $X$ and $Y$ is a complex analytic subset 
$Z \subset X \times Y$ such that the restriction $\pi_X\vert Z$ is a proper surjective generically $d$-to-one map. 
Hence, $\pi_Y \circ \pi_X^{-1}$ is defined generically on $X$ (i.e., outside a proper complex analytic subset in $X$), 
and can be viewed as a $d$-valued map. In what follows we denote a meromorphic correspondence by a triple $(Z;X,Y)$ 
equipped with the canonical projections:

\begin{center}
\begin{tikzpicture}[description/.style={fill=white,inner sep=2pt}]
    \matrix (m) [matrix of math nodes, row sep=3em,
    column sep=2em, text height=1ex, text depth=0.25ex]
    {  & Z & \\
      X & & Y \\
    };
       \path[->] (m-1-2) edge  node[auto,swap] {$ \pi_X $}  (m-2-1);
       \path[->] (m-1-2) edge node[auto] {$ \pi_Y $} (m-2-3);
\end{tikzpicture} 
\end{center}

 A {\it multiple-valued meromorphic first integral} of a singular $d$-web $\mathcal W$ in $U$ is a $d$-valued meromorphic correspondence 
$(Z;U,\mathbb CP)$ such that level sets $\pi_X \circ \pi_Y^{-1}(c)$, $c \in \mathbb CP$ are invariant subsets of $\mathcal W$, i.e., they 
consist of the leaves of $\mathcal W$.

\begin{defn}
Let $\Gamma$ be a real analytic Levi-flat hypersurface in a domain $\Omega \subset \C^n$.
We say that a holomorphic $d$-web $\mathcal W$ in $\Omega$ is the {\it extension} of the Levi foliation of $\Gamma^*$
if every leaf of the Levi foliation  is a leaf of $\mathcal W$.
\end{defn}

Although in this definition we do not require $\mathcal W$ to be irreducible, we suppose that at least one leaf of every 
component of $W$ meets $\Gamma^*$. Clearly, under this condition the singular web extending the Levi foliation is unique.

\section{Proof of Theorem~\ref{t.1}, Case $(a)$}

The basic idea of our approach is the following. The leaves of the Levi foliation can be identified with the components
of the Segre varieties associated with $\G$. It is possible to find a complex line parametrizing 
all Segre varieties of $\Gamma$. While for a general real analytic hypersurface $\G$ in $\cx^n$, the corresponding family 
of Segre varieties is $n$-dimensional, Levi-flat hypersurfaces can be characterized as those whose Segre family is 
one-dimensional, and ultimately this is the reason why the Levi foliation admits extension to the ambient space. 
Essentially, a suitably chosen one-dimensional family of Segre varieties is the meromorphic (perhaps, multiple-valued) 
first integral. Its graph can be described by a system of $n-1$ first order PDEs. This system defines an $n$-dimensional 
complex analytic subvariety of the 1-jet bundle of holomorphic functions on $\cx^{n-1}$. This subvariety can then be 
compactified in the projectivized cotangent bundle of $\cx^n$, which gives the singular $d$-web.

We assume that we are in the hypothesis of part (a) of Theorem~\ref{t.1}. The proof consist of several steps.

\subsection{Existence of a parametrizing complex line} We begin with the following

\begin{prop}\label{l.A'}
Under the assumptions of Theorem~\ref{t.1}, for a sufficiently small neighbourhood $\Omega$ of the origin
there exists a complex line $A\subset \mathbb C^n$ with the following properties:
\begin{enumerate}
\item[(i)] $A \cap Q_0  = \{0\}$;
\item[(ii)] $A \not\subset \G^{sng}$;
\item[(iii)] For every $q\in \G^*\cap \Omega$, there exists a point $w\in A$ 
such that $\mathcal L_q \subset Q_w$.
\end{enumerate}
\end{prop}

The existence of such $A$ should be compared to the transversal of the Levi foliation in the smooth case: if $\G$ is given 
by $\{{\rm Re\,} z_n =0\}$, then the complex line $\{z_1 = \dots = z_{n-1}=0\}$ intersects all Segre varieties, and 
can be used as a local parametrization both of the Levi foliation and its extension.

\begin{proof}
(i) By the assumptions of the theorem, $Q_0$ is a (possibly reducible) complex analytic hypersurface in a neighbourhood of the origin. 
Its tangent cone at the origin also has the complex codimension 1 (see \cite{Ch}). Hence, every complex line $A$ that is not contained 
in  the tangent cone to $Q_0$ at the origin satisfies~(i). 

(ii) By analyticity, a real line in $\C^n$ through the origin either is contained in $\Gamma$ or intersects it in a finite number of points. 
Therefore, since $\Gamma$ has empty interior in $\C^n$,  real lines which are not contained in $\Gamma$ form an open dense subset 
in the corresponding grassmanian. It suffices now to consider such a line $l$ and to take as $A$ the complex line containing $l$. In particular,  
$A \not\subset \G^{sng}$. Hence, (ii) also holds.

(iii) Recall that by Corollary \ref{Segre+Levi} we have ${\mathcal L}_q \subset S_q$ where $S_q$ is the unique irreducible component 
of $Q_q$ containing the point $q$; furthermore,  $S_q\subset \G$. Note also that  $S_q\setminus \G^{sng}$ is a finite union of leaves 
of the Levi foliation, one of which is $\mathcal L_q$. The finiteness of this decomposition is due to Corollary~\ref{Segre+Levi}(d) and the 
assumption that there are no dicritical singularities in $\Gamma\cap \Omega$. Since $Q_0 \cap A = \{0\}$, every complex hypersurface 
which is a small perturbation of $Q_0$, has a nonempty intersection with $A$ in view of the positivity of intersection indices for complex 
analytic sets (see, for instance, \cite{Ch}).  In particular, after shrinking $\Omega$ if necessary, we may assume that every $Q_w$ is 
defined by the global equation (\ref{Segre1}). Then $S_q$ is a small perturbation of a component of $Q_0$ and therefore has a nonempty intersection 
with $A$. 

Consider a point  $w\in S_q \cap A$. If $w \in \Gamma^*$ then $S_q = S_w \subset Q_w$ by Corollary~\ref{Segre+Levi}(c). To treat
the case $w\notin \Gamma^*$ we will need the following. 

\begin{lemma}\label{reg}
For every regular point $a \in S_q$ (i.e., such that $S_q$ is a smooth complex hypersurface near $a$) one has $S_q \subset Q_a$.
\end{lemma}

Note that the point $a$ may belong to $\Gamma^{sng}$.

\begin{proof} 
Since the complex hypersurface $S_q$ is irreducible,   the set of regular points of $S_q$ is an open connected (hence, connected by arcs) 
complex manifold (see \cite{Ch}). Consider a continuous path $\gamma: [0,1] \longrightarrow S_q$ with $\gamma(0) = q$, $\gamma(1) = a$. 
Introduce the set $I$ of $t_0 \in [0,1]$ such that $S_q \subset Q_{\gamma(t)}$ for every $t$ in a neighbourhood of $t_0$. By this definition 
$I$ is an open subset of $[0,1]$. It is nonempty because $0 \in I$ by Corollary~\ref{Segre+Levi}(c). We show that $I$ is also closed.  Consider 
a sequence of points $t^m \in I$ converging to $\hat t$. Then $S_q \subset Q_{\gamma(t^m)} = \{ z: \rho(z,\gamma(t^m)) = 0 \}$ for each $m$.
Passing to the limit we obtain that $S_q$ is contained in $Q_{\gamma(\hat t)}$. Note that the point $\hat w:= \gamma(\hat t)$ is regular for $S_q$, but in general can be singular for $\Gamma$. 

Let $z  = (z',z_n)$ with $z' = (z_1,...,z_{n-1})$. Applying the implicit function theorem, without loss of generality we may assume that $S_q$ 
is defined near $\hat w=(\hat w', \hat w_n)$ as the graph $z_n = h(z')$ of a function $h$ holomorphic in a neighbourhood of $\hat w'$ with 
$h(\hat w') = \hat w_n$. Since $S_q$ is contained in $\Gamma$, we have $\rho(z',h(z'),\overline{z'},\overline{h(z')}) \equiv  0$. Therefore, 
the function $\rho(z',h(z'),\overline{w'},\overline{h(w')})$ vanishes identically in $(z',w')$. This implies that for every point $b = (b',h(b'))$ in 
a neighbourhood of $\hat w$ in $S_q$,  the Segre variety $Q_b$ contains an open piece of $S_q$. We conclude by uniqueness that $S_q$ is 
contained in $Q_b$ for each $b \in S_q$ in a neighbourhood of $\gamma(\hat t)$. In particular, this holds for points $\gamma(t)$ with $t$  
in a small neighbourhood of $\hat t$. Therefore, $\hat t \in I$, and $I $ is also closed, that is, coincides with $[0,1]$.
\end{proof}
 
 To conclude the proof of (iii), assume that $w \in \Gamma^{sng}$. Let $w^m$ be a sequence of regular points in $S_q$ converging to $w$. 
 By Lemma \ref{reg}, $S_q \subset Q_{w^m} = \{ z: \rho(z,w^m) = 0 \}$ for every $m$. Passing to the limit we obtain that 
 $S_q \subset Q_w$.
\end{proof}

\subsection{From Segre varieties to differential equations} Let $A=\{r(z)=0\}$ be the complex line passing through the origin given by 
Proposition~\ref{l.A'}. Let $\Omega= \Omega' \times \Omega_n \subset \cx_{(z_1,...,z_{n-1})} \times \cx_{z_n}$ be a polydisc 
centred at the origin. Since by assumption $\dim_\C Q_0 = n-1$,  after a generic complex linear change of coordinates we have
\begin{eqnarray}
\label{nondegeneracy}
Q_0 \cap \{ (0',z_n) \} = \{ 0 \}
\end{eqnarray}
in a neighbourhood $\Omega$ of the origin. Shrinking  $\Omega$ if needed, we obtain that 
 for any $w\in \Omega$ the variety $Q_w \cap \Omega$ has a proper projection onto~$\Omega'$.
  We fix this  coordinate system in $\Omega$.

Consider the real analytic set 
\begin{equation}\label{e.x'}
X = \{(z,w) \in \Omega \times \Omega : \rho (z,\overline w)=0, \overline{r(w)}=0 \}.
\end{equation}
Denote by $\D = \{ \zeta \in \C: \vert \zeta \vert < 1 \}$ the unit disc in $\C$. Let 
$A:\mathbb D \ni t \to w(\bar t)=(w_1(\bar t),...,w_n(\bar t)) \in \Omega$, $w(0) = 0$,  be an anti $\C$-linear 
parametrization of the complex line $A$. Then the set $X$ uniquely defines the complex analytic set 
\begin{equation}\label{e.y'}
Y = \left\{(z,t)\subset \Omega \times \mathbb D : \hat\rho(z,t): = \rho(z,\overline{w(\bar t)})=0 \right\} .
\end{equation}
Intuitively, the set $Y$ can be thought of as the union of Segre varieties of points parametrized by the line $A$. 
Note that by the choice of $A$, the canonical projection $\tau: Y \to \Omega$, $\tau(z,t) \mapsto z$,  is a proper map in a 
neighbourhood of the origin because the fibre  $\tau^{-1}(0) \cap Y$ is discrete (note also that it contains the origin). Indeed, 
this fibre can be written in the form $\{ t: \rho(0,\overline{w(t)}) = 0 \}$ and in view of (\ref{e.x'}) is in a one-to-one correspondence 
with the set $\{ w: 0 \in Q_w, w \in A\}$. By Proposition~\ref{SegreProp}(a)  the latter set coincides with $\{ w \in A \cap Q_0 \}$ 
which is discrete by Proposition \ref{l.A'}(i) .

Therefore, by the Weierstrass Preparation Theorem, there exists a polydisc $U = U_z \times U_t \subset \Omega \times \mathbb D$ 
centred at the origin such that
\begin{eqnarray}
\label{FirstInt}
Y \cap U = \{(z,t) \in U: H(z,t) = 0 \},
\end{eqnarray}
 where $H$ is a Weierstrass polynomial in $t$:
\begin{equation}\label{e.H}
H(z,t) = t^k + h_{1}(z) t^{k-1}+ \cdots + h_{k}(z) 
\end{equation}
with the coefficients $h_j(z)$, $j=1,\dots,k$, holomorphic in $U_z$. Although in general this polynomial is reducible over the ring 
${\mathcal O}(U_z)$ of functions holomorphic in $U_z$, we can assume that it does not contain irreducible factors of multiplicity 
two or higher, as we can easily get rid of them. Indeed, consider the greatest common divisor $P$ of the polynomial $H$ and its 
derivative $\frac{d}{dt}H$. By Euclid's algorithm this is again a polynomial in $t$ with holomorphic coefficients. Then the quotient 
$H/P$ is a Weierstrass polynomial with the same divisors as $H$ but of multiplicity exactly equal to 1. This transformation does 
not affect the zero set of $H$. Hence, in what follows we assume that all divisors of $H$ are of multiplicity one, i.e., $H$
 has a decomposition
\begin{eqnarray}
\label{factors}
H(z,t) = H_1(z,t) \cdot ... \cdot H_m(z,t) ,
\end{eqnarray}
where $H_j$ are  polynomials of strictly positive degrees in $t$, irreducible over the ring  ${\mathcal O}(U_z)$, and such that
$H_k \neq H_j$ if $k \neq j$. Note also that 
$$H_j(0,0) = 0$$
for all $j$ (otherwise the factors would be invertible, and we can get rid of them by shrinking $U$). Furthermore, since 
$\{ z: H(z,0) = 0 \} = \{ z: \rho(z,0) = 0 \} = Q_0$, it follows from~\eqref{nondegeneracy} that for every $j$ the holomorphic 
function $z_n \mapsto H_j(0',z_n,0)$ does not vanish identically.

We now treat $z_n$ as a function of $z_1,...,z_{n-1}$, and denote $\frac{\partial z_n}{\partial z_j}=p_j$. For $j=1,...,n-1$, 
denote by $D_j$ {\it the total derivative operator} 
\begin{equation}\label{e.diff}
D_j = \frac{\partial}{\partial z_j} + p_j\frac{\partial}{\partial z_n}.
\end{equation}
Set $p= (p_1,...,p_{n-1})$, $z'=(z_1, \dots, z_{n-1})$, and let
\begin{equation}
G_j(z,t,p_j) := D_jH(z,t) = \frac{\partial}{ \partial z_j} H(z,t) + p_j \frac{\partial}{\partial z_n} H(z,t).
\end{equation}

Note that it follows from the above discussion that the derivative $\frac{\partial}{\partial z_n} H(z,t)$ does not vanish identically (in $z_n$). As a consequence, $G_j$ is a polynomial of degree 1 in $p_j$.

If $G_j$ is independent of $t$, then 
\begin{equation}\label{supl}
G_j(z,0,p_j)=0
\end{equation}
defines a partial differential equation  in $U$ with respect to the unknown function $z_n = z_n(z')$.
Suppose that some $G_j$ is a polynomial of strictly positive degree in  $t$. Consider the system of equations
\begin{equation}\label{e.system'}
\begin{cases}
H(z,t)=0,\\
G_j(z,t,p_j)=0
\end{cases}
\end{equation}
in variables $(z,t,p_j)$. Since both $H$ and $G_j$ are polynomials with respect to the variable $t$ and
the leading coefficient of $H(z,t)$ equals 1, a triple $(z,t,p_j)$ is a solution of system~\eqref{e.system'} 
if and only if the resultant $R(H,G_j)$ with respect to $t$ of $H$ and $G_j$ equals zero. Note that $R(H,G_j)$ is a polynomial in 
$p_j$ with holomorphic in $z$ coefficients of degree $d_j \le k$. 

\begin{lemma}\label{l.R'}
$R(H,G_j)(z,p_j) \not\equiv 0$.
\end{lemma}

\begin{proof} Since the ring ${\mathcal O}(U_z)$ is an  integral domain, the resultant $R(H,G_j)$ vanishes identically  if and only if
$H$ and $G_j$ have a non-constant common factor. Suppose, for instance, that $H_m$ divides $G_j$. We can assume that $H_m$ has 
the form (\ref{e.H}) (since in any case its leading coefficient in degree of $t$ is invertible). We have $G_j = D_jH = (D_jH_1) H_2...H_m +...+  H_1 H_2... (D_jH_m)$. Note that the polynomial $D_jH_m$ does not vanish identically because the derivative $\frac{\partial}{\partial z_n} H_m$ does not vanish identically in $z_n$.
Since $H_k$ and $H_m$ are coprime, we conclude that   $H_m$ divides $D_jH_m$. However $D_jH_m$
has a lower degree in $t$ than $H_m$.  This contradiction proves the lemma. 
\end{proof}

Letting $\Phi_j(z,p):= R(H,G_j)$ (or $\Phi_j := G_j$  if $G_j$ is independent of $t$) we obtain 
the system of partial differential equations
\begin{eqnarray}\label{ME1'}
\Phi_j(z,p_j) = 0, \ j=1,...,n-1.
\end{eqnarray}
This  system plays the key role in our approach.

\subsection{Extension of the Levi foliation and the first integral} Treating $p=(p_1,\dots,p_{n-1})$ as coordinates in $J^1(\Omega,\cx)$, system~\eqref{ME1'} defines a complex analytic 
subset $W$ of  $J^1(\Omega,\cx)\equiv \cx^n_z \times \cx^{n-1}_p$. Let $\pi: W \to \cx^n_z$ be the coordinate projection.

Our first claim is that $\pi(W)$ is an open subset of $\Omega$. Indeed, for a generic point 
$z^0 \in \Omega$ there exists a neighbourhood $U$ of $z^0$ and $t^0\in \mathbb D$ such that $z^0\in Q_{w(t^0)}$ 
and $Q_{w(t)} \cap U$ is a complex hypersurface for all $t$ close to $t_0$. After shrinking $U$ if necessary, we may 
assume that $Q_{w(t)} \cap U = \{ z_n = h_t (z')\},$ where $h_t$ is holomorphic. Then points $(z', h_t(z'), t)$ clearly 
satisfy~\eqref{FirstInt}. Applying operator $D_j$ to the identity $H((z',h_t(z')), t)\equiv0$ we obtain that the triple
$(z', h_t(z'), \frac{\partial h_t(z')}{\partial z_j})$ satisfies~\eqref{e.system'} for all $j=1,\dots, n-1$, and $t$ sufficiently
close to $t^0$. Then $z^0 \in \pi(z', h_t(z'), \frac{\partial h_t(z')}{\partial z_j})$, which shows that a generic point $z^0$
belongs to $\pi(W)$. 

Our second claim is that $\dim W = n$. Indeed, note that every $\Phi_j(z, p_j)$ has nontrivial dependence on $p_j$, as
otherwise, the set $\pi(W)\subset \{\Phi_j=0\}$, which contradicts  the previous claim. This shows that at a generic point,
the hypersurfaces $\{\Phi_j (z, p_j)=0\}$ meet in general position, and therefore, they define a complex analytic set of 
dimension $n$.

Now, observe that each $\Phi_j(z, p_j)$ has polynomial dependence on $p_j$. Indeed, each $G_j(z,t,p_j)$ depends linearly on
$p_j$, but the dependence becomes polynomial in $R(H, G_j)$. Therefore, a standard projectivization procedure defines
compactification of $W$ in $\PP T^*_n \Omega$ which we denote again by $W$ for simplicity. Thus, we obtain a complex
analytic subset of $\PP T^*_n \Omega$ of dimension $n$, which agrees with the solution of~\eqref{ME1'} in the affine
part. Its projection $\pi : W \to \Omega$ is proper with a discrete fibre at a generic point. In particular, it is surjective 
because $\pi(W)$ is a complex analytic subset of $\Omega$ of dimension $n$. Therefore, $W$ defines the required $d$-web 
in a neighbourhood $U_z$ of the origin. Note that by construction, every Segre variety $Q_w$ of a point $w \in A$ 
satisfies the system~(\ref{ME1'}). Since by Proposition~\ref{l.A'}(iii) every leaf of the Levi foliation $\mathcal L$ in $\G^* \cap U$  
is contained in the Segre variety of some point in $A$, we conclude that every leaf of $\mathcal L$ also satisfies this system.
This means the web is an extension of the Levi foliation. 

Finally, the set $Y$ defined by~\eqref{e.y'} has the proper projection $\tau : Y \to \Omega$, and this gives the first integral 
of the web given by $W$. This completes the proof of (a) in Theorem~\ref{t.1}.

\section{Proof of Theorem~\ref{t.1}, Case $(b)$}

The case when the origin is a nondicritical singularity of $\G$ is contained in part (a) of Theorem~\ref{t.1}, so 
assume that $0\in \G$ is dicritical. Since the defining function $\rho(z,\bar z)$ of $\G$ is polynomial,
every Segre variety is a (possibly reducible) algebraic variety, which after projectivization can be 
considered to be closed in $\mathbb CP^n$. By Bezout's theorem (see, for instance, \cite{Ch}), given a projective 
line $A$, which is not an irreducible component of any Segre variety of $\G$, the intersection of $A$ with any 
component of $Q_w$ is a discrete set provided that $w$ is not a dicritical singularity. For example, a 
generic complex line $A$ with $0\notin A$ has this property. Since the set of dicritical singularities of $\G$ has 
dimension at most $n-2$, after a small perturbation of $A$ 
we may assume that $A$ does not contain such singularities. Then the same proof as in Proposition~\ref{l.A'} shows 
that for a sufficiently small neighbourhood $\Omega$ of the origin, for any $q\in \Omega\cap \G^*$, 
the leaf $\mathcal L_q$ of the Levi foliation is contained in a Segre variety of some point in $A$.

Denote by $L_\infty$ the complex hyperplane at infinity in $\mathbb CP^n$, and let $a =A\cap L_\infty$. 
Let 
$$
\mathbb C \ni t\to A(\bar t)=(w_1(\bar t),..., w_n(\bar t)) \in \mathbb C^n
$$ 
be an anti $\C$-linear parametrization of the affine part of $A$. Without loss of generality we may assume that 
the component $w_n(\bar t)$ is not equal  identically to a constant. 
Then we conclude that for any $\mathcal L_q \not\subset Q_a$, there exists $t\in \mathbb C$ such that  
$\mathcal L_q \subset Q_{A(\bar t)}$. As in the proof of Theorem~\ref{t.1}, we construct the set 
$Y\subset \Omega \times \mathbb C$ given by the polynomial
$$
H(z,t) = \rho(z,\overline{w(\bar t)})=0.
$$
Note that since the origin is a dicritical singularity, the projection $\pi: Y \to \Omega$ is not proper because
$\pi^{-1}(0)$ has positive dimension; hence we are not in the hypothesis of the Weierstrass Preparation Theorem. 
But since $H (z,t)$ is already a polynomial in $t$ (and even in $z$), it still makes sense to consider
$R(H,G_j)$, the resultant of $H$ and the polynomial $G_j = D_jH$, where $D_j$ are as in \eqref{e.diff} (of course, we 
may assume that $H$ is not constant in the variable $z_n$). 
As above, we may also assume that $H$ does not contain multiple factors so that $R(H,G_j)$ does not vanish identically. 
The equalities $R(H,G_j)=0$, $j=1,...,n-1$,  give a system of partial differential equations  that defines a
$d$-web $\mathcal W$ for some $d\in \mathbb N$. By construction, all Segre varieties $Q_w$, $w\in \Omega\cap \G^*$,
satisfy this system, except possibly those that contain the point $a$ in the closure. But since $a$ is not a dicritical 
singularity, this is a finite set of Segre varieties, and it follows from the analytic dependence of Segre varieties on 
the parameter that the obtained differential equation is satisfied by all $Q_w$. In particular, this shows  
that $\mathcal W$ extends the foliation~$\mathcal L$. 

The set $Y$ defines a meromorphic correspondence  which is the first integral of~$\mathcal W$.
This proves Theorem~\ref{t.1}(b).


\section{Proof of Theorem~\ref{t.2}}

By the assumption of the theorem, $H (z,\zeta)$ admits expansion (\ref{rational}) with $\zeta \in \mathbb S$.  Complexifying in 
the variable $\zeta$, we obtain the function  $H(z,\zeta)$ defined for $(z,\zeta) \in \Omega \times \C\setminus\{0\}$ 
by (\ref{rational}). 

Consider the set 
$Y\subset \Omega \times \mathbb C\setminus \{0\}$ given by the equation $\{ (z,\zeta): H(z,\zeta) =0 \}$.
In general,  the projection $\pi: Y \to \Omega$ is not proper because the fibre 
$\pi^{-1}(0)$ can be of  positive dimension; hence, we are not in the hypothesis of the Weierstrass Preparation Theorem. 
Since the function $\hat H = \zeta^N H(z,\zeta)$ is  polynomial in $\zeta$,  we may consider the resultant 
$R(\hat H,G_j)$; as above, $G_j = D_j\hat H$, where $D_j$ are as in \eqref{e.diff}. Now we can conclude the proof  
as in the previous section. 

Note that in this construction the complex hypersurfaces $\mathcal F_\zeta := \{ z: H(z,\zeta) = 0 \}$ play the role
of the Segre varieties. At regular points of $\Gamma$ these hypersurfaces coincide with leaves of the Levi foliation.
As above, we can assume that $\hat H$ does not contain multiple factors so that $R$ does not vanish identically. 
The equalities $R(\hat H,G_j)=0$, $j=1,...,n-1$, again  give a system of partial differential equations, polynomial in $p_j$,  that defines a
$d$-web $\mathcal W$ for some $d\in \mathbb N$. By construction, all complex hypersurfaces $\mathcal F_\zeta$ satisfy this system. In particular, this shows  
that $\mathcal W$ extends the foliation~$\mathcal L$. 

The set $Y$ is algebraic in variable $\zeta$. Therefore, projectivization in the variable $\zeta$ gives the compactification
of $Y$ in $\Omega \times \cx P$. This defines a meromorphic correspondence  
which is the first integral of~$\mathcal W$. This proves Theorem~\ref{t.2}.

\section{Examples and Remarks}\label{s.examples}

\begin{ex}[Brunella~\cite{Bru1}]\label{e.B}
This example, discovered by M.~Brunella, shows that in general the Levi foliation of a Levi-flat hypersurface 
admits extension to a neighbourhood of a singular point only as a web, not as a singular
foliation. Consider the Levi-flat hypersurface
\begin{equation}\label{e.b1}
\G = \{z\in \mathbb C^2 : y_2^2=4(y_1^2+x_2)y_1^2\}.
\end{equation}
The singular locus of $\G$ is the set $\{y_1=y_2=0\}$. Its subset given by $\{y_1=y_2=0,\ x_2 <0\}$
is a ``stick", i.e., it does not belong to the closure of smooth points of $\Gamma$. The Segre varieties of 
$\G$ are given by
\begin{equation}\label{e.b2}
Q_w = \{z\in\mathbb C^2: (z_2-\bar w_2)^2+(z_1-\bar w_1)^4 - 2(z_2+\bar w_2)(z_1 -\bar w_1)^2=0\}.
\end{equation}
We see that $Q_0= \{z_2^2 +z_1^4 - 2 z_2 z_1^2=0\}$, and the origin is a nondicritical singularity. Following
the algorithm in the proof of Theorem~\ref{t.1} we choose $A(t)$ to be given by $w_1=0$, $w_2=\bar t$. 
Then~\eqref{e.b2} becomes
$$
(z_2-t)^2 + z_1^4 - 2z_1^2(z_2+t)=0.
$$
After differentiation with respect to $z_1$, and using the notation $p=\frac{dz_2}{dz_1}$ we obtain
$$
2(z_2-t)p+4z_1^3 - 4z_1(z_2+t) - 2z_1^2p=0 .
$$
Direct calculation shows that the resultant of the two polynomials in $t$ above vanishes (after dropping 
irrelevant factors) when
\begin{equation}\label{e.bb}
p^2 = 4z_2.
\end{equation}
This is the 2-web that extends the Levi foliation of $\G^*$. Its behaviour is described in Example~\ref{ex.brunella}.
Note that the exceptional set $\{z_2=0\}$ intersects $\Gamma$ along the line $\{z_2=y_1=0\}\subset \G^{sng}$.

By inspection of solutions of~\eqref{e.bb} we see that a first integral of $\G$ can be taken to be
$$f(z_1,z_2)=z_1 \pm \sqrt{z_2} ,$$ 
where $f$ is understood as a multiple-valued $1-2$ map. In fact, one can immediately verify that the closure of 
the smooth points of $\Gamma$ is given by
$$
\{z\in \mathbb C^2 : {\rm Im\,}(z_1 \pm  \sqrt{z_2}) =0 \} =
\{{\rm Im\,}(z_1 +  \sqrt{z_2})\} \cup \{{\rm Im\,}(z_1 - \sqrt{z_2})\} .
$$
However, the points of the stick cannot be recovered from the first integral. $\diamond$
\end{ex}

\begin{rem}
It is a separate question whether a Levi-flat hypersurface can be given as the preimage of a 
real analytic curve under a first integral. As discussed above, this is not the case in
Brunella's example. In fact, this is a general phenomenon of all umbrellas
(by an umbrella we mean an irreducible real analytic hypersurface that contains points near which it is a 
smooth manifold of dimension less than $2n-1$). Indeed, let $\Gamma$ be a Levi-flat hypersurface
that admits a (single or multiple-valued) meromorphic integral $f : U \to \cx P^1$ for a neighbourhood $U$ of a
singular point near which $\G$ is an umbrella. Suppose that there exists a 
connected real analytic curve $\gamma \subset f(U)$ such that $\Gamma\cap U$ can be given
as $f^{-1}(\gamma)$. 
If $a\in \G$ is a point in the stick of $\G$, and $a' \in f(a) \subset \gamma$, then all points in the stick belong 
to $f^{-1}(a')$, i.e., the stick is a complex hypersurface $f^{-1}(a')$. However, for points in $\gamma$ 
close to $a'$, their preimage under $f$ has no intersection with a neighbourhood of $a$. Hence, $a'$
is an isolated point in the image of $\Gamma$, which shows that $\G$ cannot be given as a preimage of a
curve. We will discuss this again in Example~\ref{ex.BG2} below.
\end{rem}

\begin{ex}\label{e.s}
Consider the set given in $\mathbb C^2$ by the equation
$$
\left\{{\rm Re\,}( z_1 \pm 1/\sqrt{z_2} ) =0 \right\},
$$
where $\pm$ is understood as in the previous example. We may get rid of the radical by squaring
the equation twice to obtain 
\begin{equation}\label{e.sh}
\G = \left\{ \left(|z_2|^2(z_1 + \bar z_1)^2-(z_2+\bar z_2)\right)^2 = 4 |z_2|^2 \right\}.
\end{equation}
It is easy to verify that $\G$ is an irreducible real analytic Levi flat hypersurface.
The corresponding Segre varieties are given by
$$
Q_w  = \left\{z\in \mathbb C^2 : \left (z_2 \bar w_2 (z_1 +\bar w_1) - (z_2 + \bar w_2) \right)^2 = 4 z_2 \bar w_2\right\}.
$$
Since $Q_0 = \{z_2=0\}$, the origin is a nondicritical singularity. Note that $\G^{sng} = Q_0$. Thus, the complex line
$t \to (0,\bar t)$ satisfies the requirements of Lemma~\ref{l.A'}. We obtain the following system of equations:
$$
\begin{cases}
H(z,t)=t^2(1-z_2z_1^2)^2 - 2tz_2(1+z_2z_1^2) + z_2^2=0,\\
G(z,t,p)=-2t^2(1-z_2z_1)(pz_1^2 + 2z_2z_1)-2t \left(p(1+z_2z_1^2) + z_2(pz_1^2+2z_2z_1)\right) + 2z_2 p =0.
\end{cases}
$$
Note that the coefficient $(1-z_2z_1^2)$ of $t^2$ in $H$ does not vanish at the origin, and thus, after division by  
$(1-z_2z_1^2)$, the function $H(z,t)$ defines a Weierstrass  polynomial in a neighbourhood of the origin. The resultant of 
$H$ and $R$ equals
$$
R(H,G)= -16z_1^2 z_2^3 (-p^2+4 z_2^3)(-1+z_2 z_1^2).
$$
The relevant factor $-p^2+4 z_2^3$ gives the differential equation that defines the $2$-web extending the Levi
foliation of $\G^*$:
$$
\left(\frac{d z_2}{d z_1} \right)^2=4 z_2^3 .
$$
The 1-2 map $f(z)= z_1 \pm \frac{1}{\sqrt z_2}$ is a  first integral, which is, of course,
the function that was used in the definition of $\G$. Finally note, that although $f(z)$ appears to be meromorphic, after a 
holomorphic change of coordinates in $\mathbb CP$, $f$ becomes holomorphic.
$\diamond$\end{ex}

\begin{ex}
Consider the hypersurface $\G$ given by 
$$
|z_1|^2 - |z_2|^2 = 0.
$$
Here, the origin is a dicritical singularity, but $\G$ is algebraic. The line $A=\{z_2=1\}$ is not a Segre variety of $\G$.
By Bezout's theorem, after projectivization, $A$  intersects every (projectivized) Segre variety $Q_w$, $w\ne 0$, at 
one point. The differential equation that gives the extension of the Levi foliation on $\G^*$ is given by 
$$
\frac{dz_2}{dz_1} = \frac{z_2}{z_1}
$$
with the meromorphic first integral equal to $f(z)= \frac{z_2}{z_1}$. The hypersurface $\G$ is then given by
$|f(z)|=1$.
$\diamond$\end{ex}

\begin{ex}\label{e.BG}
Consider the Levi flat hypersurface in $\cx^n$ given by 
\begin{equation}\label{e.realcone}
\Gamma = \left\{ {\rm Re\,}(z_1^2+\dots+z_n^2)=0 \right\}.
\end{equation}
According to~\cite{BuGo} any Levi flat hypersurface given by ${\rm Re\ }(z^2_1 +z^2_2 +... +z^2_n) +H(z,z)=0$,
where $H(z, z) = O(|z|^3)$, $H(z, \overline z) = \overline H (\overline z, z)$, can be transformed into~\eqref{e.realcone}
by a local biholomorphic change of coordinates at the origin. The Segre variety $Q_0 = \{z_1^2+\dots + z_n^2 =0\}$,
hence, the origin is an isolated nondicritical singularity. Choose the complex line $A(t)$ given by $t \to (0, \dots, 0, \bar t)$. 
Plugging this into the general equation for Segre varieties we obtain
$$
H(z,t)=z_1^2 + \dots + z_n^2 + t^2 =0.
$$
Treating $z_n$ as a function of $z'=(z_1, \dots, z_{n-1})$, and applying the operators $D_j$ from~\eqref{e.diff} to  
the function $H(z,t)$ we obtain equations
$$
G_j(z,p_j) = z_j + z_n p_j =0, \ \  j=1,\dots, n-1.
$$
Note that none of $G_j$ depends on $t$, so instead of taking the resultant, we simply consider the system
$G_j(z,p_j)=0$, which defines a complex analytic subset $\tilde W$ of $J_1(\cx^{n-1},\cx)$ of dimension $n$. The corresponding
compactification of $\tilde W$ in $\mathbb PT^*_n$ is given by 
$$
W = \left\{ z_j \hat p_0 + z_n \hat p_j =0, \ \ j=1,\dots, n-1 \right\},
$$
where $[\hat p_0,\dots, \hat p_{n-1}]$ are homogeneous coordinates in the compactified fibres, and $p_j = \hat p_j / \hat p_0$. 
Let $\pi: W \to \cx^n$ be the coordinate projection. Then $\pi^{-1}(0) \cong \cx P^{n-1}$, and $\pi^{-1}(z) \cong \cx P^{n-2}$
for all $z$ of the form $z=(z',0)$, $z'\ne 0$. Outside the locus $z_n=0$, the map $\pi$ is a biholomorphism. Note that the projection
$\pi$ is proper, in particular, it is surjective, and thus defines a web in $\cx^n$. In fact, the web is a foliation with an isolated 
singularity at the origin. To see this recall that the foliation on the regular part of $W$ is given by the form $\theta$ as in~\eqref{e.cont2}.
It pullback to $\cx^n$, $\theta^* = (\pi^{-1})^* \theta$, is given by 
$$
\theta^* = dz_n - \sum_{j=1}^{n-1} \left( - \frac{z_j}{z_n} \right) dz_j ,
$$
which we can simply write as $\theta^* = z_1 dz_1 + \dots + z_n dz_n$. Clearly, $\theta^*$ generates a foliation whose leaves are
complex hypersurfaces of the form $\{z_1^2+\dots+ z_n^2 +c =0\}$. These are precisely the Segre varieties of $\G$ in~\eqref{e.realcone}.
The Segre variety $Q_0$ is the only variety which is nonsmooth (at the origin), and so the origin is the only singularity of the 
foliation. 

Finally, observe that the set $W\subset \cx T^*_n$ has fibres of positive dimension even when the Levi foliation of $\Gamma$
extends as a foliation, not as a web, and that not every point in $\cx^n$ over which the fibre has positive dimension is necessarily
a singular point of the extended foliation. Further, the first integral given by $H(z,t)=0$ that we obtained following
the general algorithm is multiple-valued. However, a composition with $t \to t^2$ gives a single-valued first integral. As a result,
the extension of the Levi foliation does not branch, i.e., is a $1$-web. $\diamond$
\end{ex}

\begin{ex}[cf.~\cite{BuGo}]\label{ex.BG2}
Consider the hypersurface 
\begin{equation}
\G = \{z \in \cx^2 : x_1^2 + y_1^3 =0\}.
\end{equation} 
It is Levi flat since it is foliated by by complex lines parallel to $z_2$-axis. Burns and Gong~\cite{BuGo} showed that this Levi flat cannot be
given in the form $\{{\rm Re\,}f=0\}$, where $f$ is a meromorphic or holomorphic function.  Segre varieties of $\G$ are given by
$$
Q_w = \{2i(z_1+\bar w_1)^2 -(z_1- \bar w_1)^3 =0\}.
$$
We have $Q_0 = \{z_1=0\}\cup \{z_1=2i\}$, where the second component can be disregarded from local considerations because
it does not pass through the origin. As computations show, for points $w \in \G$ close to the origin, the Segre variety $Q_w$ consists 
of two components that are close to $\{z_1=0\}$, and only one of them is contained in $\G$. 

In this example it is natural to consider Segre varieties as graphs $z_1 = z_1(z_2)$. Let $A(t)$ be given 
by $t \to (\bar t, 0)$. Then
\begin{equation}\label{e.HH}
H(z,t) = t^3 + t^2(2i-3z_1) + t(3z_1^2 + 4iz_1) + (2iz_1^2-z_1^3).
\end{equation}
Applying the total derivative operator to $H$ we obtain a polynomial $G(z,t,p)$ with the property that all of its monomials have degree 
one in $p$. This means that $R(H,G)$, the resultant of $H$ and $G$, factors out a power of $p$. Therefore, $R(H,G)$ vanishes if $p=0$. 
This gives a trivial extension of the Levi foliation on $\G$ with the leaves of the form $\{z_1={\rm const}\}$, with the corresponding 
first integral given by $f(z) = z_1$. 

From this we see that $\Gamma = f^{-1}(\gamma)$, where $\gamma = \{ t = u+iv \in \cx: u^2 + iv^3 =0\}$. For comparison, 
the first integral given by equation~\eqref{e.HH} is a $1-3$ map, which by  construction attains the same value on all components of $Q_w$. 
Because for $w\in\G$, not all components of $Q_w$ are contained in $\Gamma$, the preimage of $\gamma$ by the map given
by~\eqref{e.HH} contains more points in $\cx^2$ than $\Gamma$.  
$\diamond$\end{ex}

We further remark that for a first integral $f$ constructed from the complex curve $A(t)$ as in the proof of Theorem~\ref{t.1}, the
Levi flat $\G$ cannot be in general given in the form $\G = f^{-1}(\gamma)$ for a real analytic curve $\gamma$. If for a generic
point $w\in \G$, the Segre variety $Q_w$ contains several components and not all of them are contained in $\G$, then the set
$f^{-1}(\gamma)$ is only subanalytic (or semialgebraic in the algebraic category), but it will contain $\G$ as a component. 
The curve $\gamma$ can be taken to be the preimage of $A\cap \Gamma$ under the parametrization of $A(t)$.

\begin{rem}
Our result should be also compared with Brunella~\cite{Bru1}, where he
 gives extension of $\mathcal L$ by considering the map from the regular part of $\G$ into its projectivized cotangent bundle $PT^*\G$. 
 It is defined by sending a smooth point $p$ in $\G$ to its complex tangent. The image of $\G$ under this map gives a $2n-1$-dimensional 
 real analytic subset of 
 $\PP T^*\C^n$, the projectivized cotangent bundle of $\mathbb C^n$. The crucial step in his construction is to show that this real analytic
 set is contained in a complex analytic subset  of $\PP T^*\C^n$ of dimension $n$. This set is obtained from abstract set-theoretical 
 considerations, and in general it is not clear whether it has a good projection into $\mathbb C^n$. Without this, Brunella's approach 
 only gives an extension of the Levi foliation on $\Gamma$ by considering some complex manifold $Y$ of dimension $n$, a Levi flat 
 hypersurface $N \subset Y$, and a holomorphic map $\pi : Y \to \mathbb C^n$ whose restriction to some open subset of $N$ is a proper 
 map onto the closure of smooth points in $\G$. This construction does not immediately give extension of $\mathcal L$ as a web to a 
 neighbourhood of a singular point of $\G$ in $\mathbb C^n$.
 \end{rem}


\end{document}